\documentclass[14pt]{amsart}

\usepackage{amssymb}
\usepackage{amstext}
\usepackage{amsmath}
\usepackage{amscd}
\usepackage{latexsym}
\usepackage{amsfonts}
\usepackage{color}
\usepackage{enumerate}
\usepackage[all]{xy}
\usepackage{graphicx}

\theoremstyle{plain}
\newtheorem{thm}{Theorem}[section]
\newtheorem*{thm*}{Theorem}
\newtheorem*{cor*}{Corollary}
\newtheorem*{defn*}{Definition}

\newtheorem{prop}[thm]{Proposition}
\newtheorem{lem}[thm]{Lemma}

\newtheorem{claim}[thm]{Claim}
\newtheorem*{claim*}{Claim}

\usepackage[pagebackref]{hyperref}
\usepackage{hyperref}
\hypersetup{
    colorlinks=true, 
    linkcolor=red, 
}

\usepackage[alphabetic,initials]{amsrefs}

\theoremstyle{definition}

\newtheorem{rem}[thm]{Remark}

\newtheorem{nota}[thm]{Notation}

\theoremstyle{remark}

\def\pr{\mathrm{pr}}

\begin{document}
\title[]{On cylindrical smooth rational Fano fourfolds }

\author[N.T.A. Hang]{Nguyen Thi Anh Hang}
\address{The Department of Mathematics, Thai Nguyen University of education.
20 Luong Ngoc Quyen Street, Thai Nguyen City, Thai Nguyen Province, Viet Nam.}
\email{hangnthianh@gmail.com}

\author[M. Hoff]{Michael Hoff}
\address{	Mathematik und Informatik, Universit\"{a}t des Saarlandes, Campus E2 4,
	D-66123 Saarbr\"{u}cken, Germany}
\email{hahn@math.uni-sb.de}

\author[H.L. Truong]{Hoang Le Truong}
\address{	Mathematik und Informatik, Universit\"{a}t des Saarlandes, Campus E2 4,
	D-66123 Saarbr\"{u}cken, Germany}
\address{Institute of Mathematics, VAST, 18 Hoang Quoc Viet Road, 10307
Hanoi, Viet Nam}
\address{Thang Long Institute of Mathematics and Applied Sciences, Hanoi, Vietnam}
\email{hoang@math.uni-sb.de\\hltruong@math.ac.vn\\
	truonghoangle@gmail.com}

\thanks{
M.H. was partially supported by the Deutsche Forschungsgemeinschaft (DFG, German Research Foundation) - Project-ID 286237555 - TRR 195. 
H.L.T.  was partially supported by the Alexander von Humboldt Foundation and the Vietnam National Foundation for Science and Technology Development (NAFOSTED) under grant number 101.04-2019.309. }

\begin{abstract} 
We construct new families of smooth Fano fourfolds with Picard rank $1$ which contain open $\Bbb A^1$-cylinders, that is, Zariski open subsets of the form $Z \times \Bbb A^1$, where $Z$ is a quasiprojective variety. In particular, we show that every Mukai fourfold of genus $8$ is cylindrical and there exists a family of cylindrical Gushel-Mukai fourfolds. 
\end{abstract}

\maketitle

\section{introduction}

A smooth complex projective variety $X$ 
 is called {\it cylindrical} if it contains a cylinder, that is, a principal Zariski open subset $U$ isomorphic to a product $Z \times \Bbb A^1$, where $Z$ is a variety and $\Bbb A^1$ is the affine line over $\Bbb C$. Complex projective varieties containing cylinders have recently started to receive a lot of attention in connection with the study of unipotent group actions on affine varieties.   

We shortly summarize what is known for low dimensional varieties. There are only two cylindrical smooth complex curves, namely the affine line $\Bbb A^1$ and the projective line $\Bbb P^1$. A smooth del Pezzo surface of degree $d$ has an anticanonically polarized cylinder  if and only if $d \ge 4$ (see  \cite{KPZ11, KPZ14} and \cite{CPW16}). In higher dimensions, several families of  Fano varieties of dimension $3$ and $4$ and Picard number one admitting (anti-canonically polar) cylinders have been constructed in \cite{KPZ14,PrZ16,PrZ17}, but a complete classification is still far from being known. 

In this paper, we focus on the case where the variety $X$ is a \emph{Fano fourfold}, that is, a four dimensional complex smooth Fano variety with ample anticanonical divisor $-K_X$.  
Given a smooth Fano fourfold $X$ with Picard rank $1$, the \emph{index} of $X$ is the integer $r$ such that $-K_X \sim rH$, where $H$ is the ample divisor generating the Picard group ${\mathrm{Pic}}(X) = \Bbb Z H$. Its \emph{degree} $d = \deg X$ is defined with respect to this ample divisor $H$. It is known that the index is within the range $1 \le r \le 5$. Moreover, if $r =5$, then $X \cong \Bbb P^4$, and if $r =4$, then $X$ is a quadric in $\Bbb P^5$. Smooth Fano fourfolds of index $r = 3$ are called del Pezzo fourfolds; their degrees vary in the range $1 \le d \le 5$. Smooth Fano fourfolds of index $r = 2$ are called \emph{Mukai fourfolds}; their degrees are even and can be written as $d = 2g - 2$, where $g$ is called the \emph{genus} of $X$ (see also Section \ref{notationAndFamilies}) and the possible values for the genus are $2\le g \le 10$. There is no classification known for Fano fourfolds of index $r=1$.

For a recent and elaborated overview of cylindrical varieties, we refer to \cite{CPPZ20}. Note that by \cite[Corollary 1.6]{CPPZ20}, any smooth cylindrical Fano fourfold of Picard rank $1$ is rational, but only a few examples are known. 
According to  \cite[Theorem 1.1]{PrZ16}, a smooth intersection of two quadrics $W_{2,2}$ in $\Bbb P^6$ is a cylindrical del Pezzo fourfold of degree $4$. In addition, a smooth del Pezzo fourfold $W_5 \subset \Bbb P^7$ of degree $5$ is also cylindrical. Starting with smooth intersections of two quadrics $W_{2,2}$ or del Pezzo quintic fourfolds $W_5$, and performing suitable Sarkisov links, Prokhorov and Zaidenberg construct in \cite{PrZ16,PrZ17} four families of cylindrical Mukai fourfolds of genus $g=7, 8, 9$ and $10$, respectively. In \cite{PrZ20}, the authors show that every Mukai fourfold of genus $10$ is cylindrical. 

We proceed in a similar fashion as in \cite{PrZ16,PrZ17} and connect Mukai fourfolds to the projective space $\Bbb P^4$ via Sarkisov links. By a careful analysis of these Sarkisov links, we provide further examples of cylindrical, smooth, rational Fano fourfolds. Note that cylindrical Mukai fourfolds of genus $6$, so-called \emph{Gushel-Mukai} fourfolds, were not known before.

\begin{thm}\label{Main}
The following four families parametrize cylindrical, smooth, rational Fano fourfolds:

\begin{enumerate}[$i)$]
\item the Gushel-Mukai fourfolds of genus $6$ containing a $\tau$-quadric surface, 
\item the Mukai fourfolds of genus $7$ containing a cubic scroll surface,
\item the Mukai fourfolds of genus $8$,
\item the Mukai fourfolds of genus $9$  containing a del Pezzo surface of degree $6$. 
\end{enumerate}
\end{thm}

\begin{rem}
 \begin{enumerate}[a)]
  
   \item The codimension of the families $i),ii)$ of Theorem \ref{Main} is one (see Section \ref{tau} also for the definition of a $\tau$-quadric surface and Section \ref{geometryCubicThreefold}.) 
    \item The family $iv)$ of Theorem \ref{Main} was already studied in \cite[Theorem 2.1 and Corollary 2.4]{PrZ17}. Indeed, a general Mukai fourfold of genus $9$ containing a sextic del Pezzo surface also contains a smooth quadric surface (see also end of Section \ref{geometryCubicThreefold}). 
    \item 
In \cite[Theorem 1.1]{PrZ16}, the authors showed the existence of cylindrical Mukai fourfolds of genus $7$ and $8$ varying in a family inside their moduli spaces of codimension $2$ and $1$,  respectively. 
Therefore, our families of cylindrical Mukai fourfolds of genus $7$ and $8$ described in Theorem \ref{Main} are new. 
 \end{enumerate}
\end{rem}

In Section \ref{notationAndFamilies}, we give an overview of the basic results and describe some families appearing in Theorem \ref{Main} in more detail. Section \ref{SarkisovLinks} is devoted to the geometry of the Sarkisov links. We show that the complement of a cubic hypersurface section of a Mukai fourfold as in Theorem \ref{Main} is isomorphic to the complement of a cubic threefold $W_3\subset \Bbb P^4$. In Section \ref{geometryCubicThreefold}, we describe the geometry of the cubic threefolds $W_3$ and their singularities depending on the genus $g$ of the Mukai fourfolds. The existence of a cylinder in the complement of cubic threefolds $W_3$ with isolated singularities and the proof of Theorem \ref{Main} are presented in Section \ref{CylinderCubics}.

\section{Notation and preliminaries}\label{notationAndFamilies}

\begin{nota}
We work over the complex numbers $\Bbb C$. 
Let $Y\subset\Bbb P^N$ be a smooth projective variety of dimension $n$ and
	\begin{enumerate}[-]

\item $H_Y$, its hyperplane class,
\item $d(Y)=H^n_Y$, its degree,
\item $K_Y$, its canonical class,
\item $\pi(Y)=\frac{1}{2}H_Y^{n-1}((n-1)H_Y+K_Y)+1$, its sectional genus,
\item $\mathcal{T}_Y$ the tangent bundle of $Y$,
\item $c_i(\mathcal E)$ the $i$-th Chern class of a vector bundle $\mathcal E$ on $Y$, and $c_i(Y) = c_i(\mathcal T_Y)$,
\item $\chi_Y = \chi(\mathcal O_Y)$, its Euler-Poincar\'{e} characteristic,
\item $\langle Y \rangle$, the linear span of $Y$. 
\item $X_{2g-2}$, a Mukai fourfold of index $2$ and genus $g$. 
	\end{enumerate}

\end{nota}

The following lemma is a special case of the Riemann-Roch Theorem.

\begin{lem}{\cite[Lemma 2.1]{PrZ16}} For a smooth projective fourfold $X$ and a divisor $D$ on $X$ we have
$$\chi(\mathcal O_X(D))=\frac{1}{24}(D^4+2D^3\cdot c_1(X)+D^2(c_1^2(X)+c_2(X))+D\cdot  c_1(X)\cdot c_2(X))+\chi(\mathcal O_X).$$
\end{lem}

\begin{lem}{\cite[Lemma 2.3]{PrZ16}}\label{chow}
Let $X$ be a smooth projective fourfold, let $\rho : \widetilde{X} \to X$ be a blowup along
a smooth surface $S \subset X$, and let $E = \rho^{-1}(S)$ be the exceptional divisor. Then
$$c_2( \widetilde{X})=\rho^\ast c_2(X)+\rho^\ast S+ \rho^\ast K_X\cdot E,$$
$$(\rho^\ast H)^4=H^4,\quad (\rho^\ast H)^3\cdot E=0, \quad (\rho^\ast H)^2\cdot E^2=-S\cdot H^2,$$
$$(\rho^\ast H)\cdot E^3= - H|_S\cdot K_S+ K_X\cdot H\cdot S,$$
$$E^4=c_2(X)\cdot S+K_X|_S\cdot K_S-c_2(S)-K_X^2\cdot S.$$
\end{lem}

We will need the above lemma for blowups of $\Bbb P^4$ along an irreducible surface with ordinary double points. 

\begin{lem}{\cite[Lemma 2.1]{HaL18}}\label{chow1}
Let $\rho : \widetilde{\Bbb P^4} \to \Bbb P^4$ be a blowup along
an irreducible surface $F$ with $\delta$ transverse double points, and let $E = \rho^{-1}(F)$ be the exceptional divisor.  Let $S$ denote the normalization of $F$,  $K_S$ its canonical class, $C$ its general sectional curve and $d=\deg F$. Let $H$ denote the divisor of the hyperplane class of $\Bbb P^4$. 
Then
$$(\rho^\ast H)^3E=0, \quad (\rho^\ast H)^2E^2=-d, \quad (\rho^\ast H) E^3= - 5d-K_S C, \text{ and } E^4=d^2-25d-10K_S C-K_S^2+4\delta.$$
\end{lem}

The next lemma follows from the corresponding result in the three-dimensional case, see \cite[Prop. 1]{Muk89}, \cite{IsP99, PCS05} and \cite[Lem. (2.8)]{Isk77}.

\begin{lem}\label{Mukaig}
Let $X_{2g-2}$ be a Mukai fourfold of genus $g  \ge 4$ with at
worst terminal Gorenstein singularities and with ${\mathrm{rank}}\  {\mathrm{Pic}} X_{2g-2} = 1$. Assume
that the linear system $|-\frac{1}{2} K_{X_{2g-2}}|$ is base point free. Then the divisor $|-\frac{1}{2} K_{X_{2g-2}}|$ is very ample and defines an embedding $X_{2g-2}\to \Bbb P^{g+2}$.

\end{lem}

\subsection{Gushel-Mukai fourfolds}
Let us recall some general facts about the theory of Gushel-Mukai fourfolds which have been proved in \cite{DIM15}. A smooth Gushel-Mukai fourfold is a smooth dimensionally transverse intersection
$$\rm{CGr(2,5)}\cap \Bbb P^8\cap Q$$
of the cone over the Grassmannian ${\mathrm{Gr}}(2, 5)$ of $2$-dimensional subspaces in a fixed $5$-dimensional vector space, with a linear subspace $\Bbb P^8$ and a quadric $Q$.   This class of varieties includes all smooth prime Fano fourfold $X = X_{10}$ of index $2$, genus $6$ and degree $10$ (that is, such that there is an ample class $H$ with ${\mathrm{Pic}}(X) = \Bbb Z H$, $K_{X} = -2H$, and $H^4 = 10$). There are two types of Gushel-Mukai fourfolds:
\begin{enumerate}[$\bullet$]
\item quadratic sections of hyperplane sections of ${\mathrm{Gr}}(2,5) \subset \Bbb P^9$ (Mukai or ordinary fourfolds, \cite{Muk89});
\item double covers of ${\mathrm{Gr}}(2, 5) \cap \Bbb P^7$ branched along its intersection with a quadric (Gushel fourfold, \cite{Gus82}).
\end{enumerate}

Let $\mathcal{M}_4^{GM}$ be the irreducible moduli space for (smooth) prime Fano fourfolds of degree $10$ and index $2$. Following \cite{Has00}, one says that a fourfold $X\in \mathcal{M}_4^{GM}$ is \emph{Hodge-special} if the rank of the (positive definite) lattice $A(X) = H^{2,2}(X) \cap H^4(X, \Bbb Z)$ is at least $3$ (see \cite[Section 6.2]{DIM15} for precise definitions). This means that the fourfold contains a surface whose cohomology class “does not come” from
the Grassmannian ${\mathrm{Gr}}(2,5)$. Hodge-special Gushel-Mukai fourfolds are parametrized by a countable union of hypersurfaces 
$$
\bigcup_d \mathcal{GM}_d \subset \mathcal{M}_4^{GM}
$$
for $d\equiv 0, 2$ and $4$ (mod $8$) where $d$ is the discriminant of the rank $3$ lattice $A(X)$. 
 Any (smooth) Gushel-Mukai  fourfold $X$ is unirational (\cite[Proposition 3.1]{DIM15}). The question of their rationality is not settled and very alike the case of cubic fourfolds in $\Bbb P^5$:  it is expected that a very general GM fourfold is not rational, but not a single example is known. There are families of (smooth) rational Gushel-Mukai fourfolds, so that it is natural to ask wether cylindrical Gushel-Mukai fourfolds exist.

There are two classical (codimension 1) families of rational Hodge-special Gushel-Mukai fourfolds, and we show that the general member in one of these families contains an $\Bbb A^1$-cylinder. Let $\mathcal{GM}_{10}$ be the Hodge-special Gushel-Mukai fourfolds of discriminant $10$ which is the union of two irreducible hypersurfaces, say $\mathcal{GM}_{10}'$ and $\mathcal{GM}_{10}''$ described as follows:  

\begin{enumerate}\label{tau}
\item  $\mathcal{GM}_{10}'$: a \emph{$\tau$-quadric} surface  in ${\mathrm{Gr}}(2,5)$ is a linear section of ${\mathrm{Gr}}(2, 4)\subset {\mathrm{Gr}}(2, 5)$.  In \cite[Proposition 7.4]{DIM15}, it is shown that the closure of the family of Gushel-Mukai fourfolds containing a $\tau$-quadric is the hypersurface $\mathcal{GM}_{10}'$. This family is studied in Theorem \ref{Main}.

\item $\mathcal{GM}_{10}''$: A \emph{quintic del Pezzo} surface is a two dimensional linear section of ${\mathrm{Gr}}(2,5)$. In \cite[Proposition 7.7]{DIM15}, it is shown that the closure of the family of Gushel-Mukai fourfolds containing a quintic del Pezzo surface is the hypersurface $\mathcal{GM}_{10}''$.
\end{enumerate}

\begin{rem}
The only further known family of rational Gushel-Mukai fourfolds is $\mathcal{GM}_{20}$ which is the closure of Gushel-Mukai fourfolds containing the image of $\Bbb P^2$ by the linear system of quartic curves through three simple points and one double point in general position (see \cite{HoS20}). 
\end{rem}

\subsection{Fano fourfold $X_{14} \subseteq \Bbb P^{10}$ of index $2$ and genus $8$}

Any Fano fourfold $X_{14} \subseteq \Bbb P^{10}$ of index $2$ and genus $8$ with ${\mathrm{Pic}}(X_{14}) \cong \Bbb Z $ is a section of the Grassmannian ${\mathrm{Gr}}(2, 6)$ under its Pl\"{u}cker embedding in $\Bbb P^{14}$ by a linear subspace of dimension $10$. 
From \cite[16.2]{BoH58} we have the formula of the total Chern class of ${\mathrm{Gr}}(2,6)$. Arranging the Schubert cycles in a triangle as follows will make the  remaining calculations more convenient
$$\begin{pmatrix}
1&&&&\\
\sigma_1&\sigma_{1,1}&&&\\
\sigma_2&\sigma_{2,1}&\sigma_{2,2}&&\\
\sigma_3&\sigma_{3,1}&\sigma_{3,2}&\sigma_{3,3}&\\
\sigma_4&\sigma_{4,1}&\sigma_{4,2}&\sigma_{4,3}&\sigma_{4,4}\\
\end{pmatrix}.
$$
If we only write down the coefficients of the corresponding Schubert cycles, the total Chern class is
$$c({\mathrm{Gr}}(2,6))=\begin{pmatrix}
1&&&&\\
6&18&&&\\
16&58&67&&\\
26&91&120&65&\\
31&90&105&60&15\\
\end{pmatrix}.
$$
The adjunction formula for the total Chern classes yields
  $c(X_{14})=\frac{c({\mathrm{Gr}}(2,6))}{(1+\sigma_1)^4}$. Thus, we have
$$c(X_{14})=\begin{pmatrix}
1&&&&\\
2&4&&&\\
2&2&5&&\\
2&0&-2&5&\\
2&-2&7&-18&27\\
\end{pmatrix}.
$$
 Therefore, $c_0(X_{14})=1$, $c_1(X_{14})=2\sigma_1$, $c_2(X_{14})=2\sigma_2+4\sigma_{1,1}$, and $c_4(X_{14})=2\sigma_4+5\sigma_{2,2}$. Then the topological Euler number
$$\chi(X_{14})=c_4(X_{14})\cdot [X_{14}]=(2\sigma_4+5\sigma_{2,2})|_{X_{14}}\cdot \sigma_1^4|_{X_{14}}\cdot[{\mathrm{Gr}}(2,6)]=12.$$
Note that the groups $H^q(X_{14},\Bbb Z)$ vanish if $q$ is odd, $H^2(X_{14},\Bbb Z) \cong H^6(X_{14},\Bbb Z) \cong \Bbb Z$, and ${\mathrm{rank}} \  H^4 ( X_{14},\Bbb Z ) = 2$.  The intersection form on $\lambda_G := H^4({\mathrm{Gr}}(2,6),\Bbb Z)|_{X_{14}}$ has matrix  in the basis 
$$(\sigma_{1,1}|_{X_{14}},\sigma_{2}|_{X_{14}}).$$
 A quintic del Pezzo surface can be obtained as a linear section of ${\mathrm{Gr}}(2, 5)\subset{\mathrm{Gr}}(2,6)$;
its class  is $$\sigma_{1,1}|_{X_{14}}=\sigma_1^4\sigma_{1,1}
 =3\sigma_{4,2}+2\sigma_{3,3}$$ in ${\mathrm{Gr}}(2, 6)$. 
 Thus, $X_{14}$  contains a quintic del Pezzo surfaces $\Sigma$.  
 
The rationality of $X_{14}$ has  been shown by  Xu  in \cite[Theorem 2.2.1]{Xu11}. Note that the following Proposition \ref{pro4}  gives a different proof of the rationality. 

\section{Birational maps to Mukai fourfolds}\label{SarkisovLinks}

In this section, we give explicit constructions of the smooth rational Mukai fourfolds $X_{2g-2}$ of genus $g=6$, $7$, $8$ and $9$ via Sarkisov links. The proof is standard and similar to \cite[Proposition 3.1]{PrZ16}, but we recall the proof for the reader's convenience.


\begin{prop}\label{pro4}
Let $ X_{2g-2} \subseteq \Bbb P^{g+2}$ be a Mukai fourfold of genus $g=6$, $7$, $8$, $9$ with ${\mathrm{Pic}}(X ) \cong \Bbb Z L$, and let $L$ be the hyperplane class of $X_{2g-2} $. 
Suppose that $X_{2g-2}$ contains   a smooth  surface $\Sigma$  as in Table \ref{S95}. Then we have the following statements. 
\begin{enumerate}[$i)$]
\item The linear system $|L - \Sigma|$  of hyperplanes passing through $\Sigma$ defines a birational map 
$$\Phi:X_{2g-2} \dashrightarrow  \Bbb P^{4}.$$

\item There is a commutative diagram
$$\xymatrix{&D\ar@{^{(}->}[r]\ar[ld]&\widetilde{X_{2g-2}}\ar[dl]_{\varphi}\ar[dr]^{\rho}&E\ar@{_{(}->}[l]\ar[rd]&\\
\Sigma\ar@{^{(}->}[r]&X_{2g-2}\ar@{-->}[rr]^{\Phi}&&\Bbb P^4&F,\ar@{_{(}->}[l]}$$
where $\varphi$ is the blowup of $\Sigma$ with exceptional divisor $D$ and $\rho$  is a birational morphism defined by the linear system $|\varphi^\ast L - D|$.

\item The $\rho$-exceptional locus is an irreducible divisor $E\subset \widetilde{X_{2g-2}}$.
\item Let  $H$ be the ample generator of ${\mathrm{Pic}}(\Bbb P^4)$. Then
$$\begin{aligned}
\varphi^\ast L \sim 4\rho^\ast H- E,\quad\quad\quad\quad&\quad\quad D\sim 3\rho^\ast H-E,\\
\rho^\ast H \sim \varphi^\ast L- D, \quad\quad\quad\quad&\quad\quad E\sim 3\varphi^\ast L-4D.
\end{aligned}
$$

\item The image $\rho(E)$ is  a surface $F\subset \Bbb P^{4}$ with at most isolated singularities as in the Table \ref{S95}. 
\item $\rho(D)$ is a cubic threefold in $\Bbb P^4$ containing $F$ and $\varphi(E)$ is a cubic hypersurface section of  $X_{2g-2}$ singular along $\Sigma$.


\item $X_{2g-2}\backslash \varphi(E)\cong \Bbb P^4\backslash \rho(D)$.
\item The cubic threefold $\rho(D)$ is singular and its singular locus coincides with the locus of points $p \in \rho(D)$, such that the restriction $\rho|_D:D\to \rho(D)$ is not an isomorphism over $p$.

\end{enumerate}

\end{prop}

\begin{table}[h!]
\centering
		\begin{tabular}{ |r| c| c| c|c|c|c|c|c|c|}
	\hline
	 $g$& $\Sigma\subset X_{2g-2}$&$d(\Sigma)$ &  $\pi(\Sigma)$ &$d(F)$&$\pi(F)$&$\#$Sing $F$&$(\varphi^\ast L)^2\cdot D^2$&$(\varphi^\ast L)\cdot D^3$&$D^4$\\ 
	\hline
	$6$&$\tau$-quadric surface &$2$   & $0$   &$8$ & $6$ &$1$&$-2$ &$0$ &$3$ 	\\ 
	\hline
	$7$&cubic scroll surface &$3$   & $0$   &$7$ & $3$ &$3$ &$-3$ &$-1$ &$3$ 	\\ 
	\hline
	$8$&quintic  del Pezzo &$5$   & $1$   &$7$ & $4$ &$0$&$-5$ &$-5$ &$-3$ 	\\ 
	\hline
	$9$&sextic  del Pezzo &$6$   & $1$   &$6$ & $1$ &$3$ &$-6$ &$-6$ &$-3$ 	\\ 
	\hline
	
\end{tabular}
\caption{Invariants of the pair of surfaces $\Sigma \subset X_{2g-2}$ and $F\subset \Bbb P^4$ depending on the genus $g$, where $\pi(\Sigma)$ and $\pi(F)$ is the sectional genus of the surface $\Sigma$ and $F$, respectively.}
\label{S95}
\end{table}

\begin{rem}
By Proposition \ref{pro4} $viii)$, the cubic threefold $\rho(D)\subset \Bbb P^4$ is always singular. 
In Section \ref{geometryCubicThreefold}, we will describe the geometry and singularities of the cubic threefolds and the surfaces $F$ with invariants as in Table \ref{S95} depending on the genus $g\in \{6,7,8,9\}$. 
\end{rem}

\begin{proof}[Proof of Proposition \ref{pro4}.]
To $i)$ and $ii)$: Since the intersection of $X_{2g-2}$ and the $(g-3)$-plane $\langle\Sigma\rangle$ is $X_{2g-2}\cap \langle\Sigma \rangle= \Sigma$, 
the linear system $|\varphi^\ast L - D|$ is base point free. Thus, the divisor $$-K_{\widetilde{X_{2g-2}}} = -K_{X_{2g-2}} - D =  \varphi^\ast L+\varphi^\ast L - D$$ is ample, that is, $\widetilde{X_{2g-2}}$ is a Fano fourfold whose rank of ${\mathrm{Pic}}(\widetilde{X_{2g-2}})$ is two.  
By the Cone theorem, there exists a Mori contraction $\rho: \widetilde{X_{2g-2}} \to U$ different from $\varphi$. 
 \\
By Lemma \ref{chow}, we have $(\varphi^\ast L)^4=2g-2$, and $(\varphi^\ast L)^3\cdot D=0$, $(\varphi^\ast L)^2\cdot D^2$, $(\varphi^\ast L)\cdot D^3$ and $D^4$ are as in Table \ref{S95}. Therefore, we have  $(\varphi^\ast L-D)^3\cdot (3\varphi^\ast L-4D)=0$ and  $(\varphi^\ast L-D)^4 = 1$. 
This means that the divisor class of $\varphi^\ast L-D$ is not ample, and so it yields a supporting linear function of the extremal ray generated by the curves in the fibers of $\varphi$. Moreover, we can write $\varphi^\ast L-D=\rho^\ast H$, where $H$ is the ample generator of ${\mathrm{Pic}}(U) \cong \Bbb Z$. 
By the Riemann-Roch theorem and Kodaira vanishing theorem, we have $\dim |\varphi^\ast L-D|= 4$.  Thus, 
$\varphi^\ast L-D$ defines a birational morphism $\widetilde{X_{2g-2}} \to \Bbb P^4$ which coincides with the map $\rho$. The birationality of $\Phi$ follows. 

To $iii)$ and $iv)$: Since $\dim |3\varphi^\ast L-4D| = 0$ and $(\varphi^\ast L-D)^3\cdot (3\varphi^\ast L-4D)=0$, the linear system
$|3\varphi^\ast L-4D|$ contains a unique divisor $E$ contracted by $\rho$. Since rank of ${\mathrm{Pic}}(\widetilde{X_{2g-2}})$ is two, the
divisor $E$ is irreducible and the relations in $iv)$ follow.


To $v)$: Since $(\varphi^\ast L-D)^2\cdot E^2 = -d(F)$, the image $F=\rho(E)$ is a {\it surface} with $\deg F = H^2\cdot \rho(E) = d(F)$ (see Table \ref{S95}). Furthermore, the sectional genus of $F$ as given in Table \ref{S95} can be computed using Lemma \ref{chow1} or \cite[Lemma 2.2.14]{IsP99}. 
Since ${\mathrm{rank}}\ {\mathrm{Pic}} (\widetilde{X_{2g-2}})= 2$, the exceptional locus of $\rho$ coincides with $E$, and
$E$ is a prime divisor. Therefore, $\rho$ has at most a finite number of $2$-dimensional fibers. By the main theorem in \cite{AnW98}, $F$ has at most isolated singularities. We will describe the geometry and the singularities of $F$ in Section \ref{geometryCubicThreefold}.

To $vi)$: From the relations $E\sim 3\varphi^\ast L-4D$ and  $D\sim 3\rho^\ast H-E$ in ${\widetilde{X_{2g-2}}}$, we deduce that the images $\varphi(E)$ and $\rho(D)$ are a cubic hypersurface sections of $X_{2g-2}$ and $\Bbb P^4$, respectively. Furthermore, $\varphi(E)$ is singular along $\Sigma$.

%
%

To $vii)$:  Since $F \subset \rho(D)$ we have isomorphisms
$$X_{2g-2}\backslash \varphi(E)\cong {\widetilde{X_{2g-2}}}\backslash(E\cup D)\cong {\Bbb P^4}\backslash(F\cup \rho(D)) \cong \Bbb P^4\backslash \rho(D).$$

To $viii)$: The cubic threefold $\rho(D)$ is singular because it contains the surface
$F$ of degree $\ge 6$ and of sectional genus as in Table \ref{S95}. Since $-K_D \sim 2(\varphi^\ast L - D)|_D \sim \rho^*(2H)|_D$, the restriction $\rho|_D : D \to \rho(D)$ is a crepant morphism.

\end{proof}


We have the following inverse picture to Proposition \ref{pro4} which includes nice geometric features. 

\begin{prop}\label{main10}
Let $F\subset \Bbb P^4$ be a surface as in Table \ref{S95} contained in a unique cubic threefold $W_3$. Then we have the following statements. 
\begin{enumerate}[$i)$]
\item The linear system $|4H - F|$ of quartic hypersurfaces passing through $F$ defines a birational map 
$$\Psi:\Bbb P^4 \dashrightarrow X_{2g-2} \subset \Bbb P^{g+2},$$
where $X_{2g-2} =\Psi(\Bbb P^4)$ is a Mukai fourfold of genus $g$ with ${\mathrm{Pic}}(X_{2g-2} ) \cong \Bbb Z$.

\item There is a commutative diagram
$$\xymatrix{&E\ar@{^{(}->}[r]\ar[ld]&\widetilde{\Bbb P^4}\ar[dl]_{\rho}\ar[dr]^{\varphi}&D\ar@{_{(}->}[l]\ar[rd]&\\
F\ar@{^{(}->}[r]&\Bbb P^4\ar@{-->}[rr]^{\Psi}&&X_{2g-2}&\Sigma,\ar@{_{(}->}[l]}$$
where $\rho$ is the blowup of $F$ with exceptional divisor $E$  and  $\varphi$  is a birational morphism defined by the linear system $|4\rho^\ast H - E|$ with $\varphi$-exceptional divisor $D$.

\item $\rho(D)=W_3$  is a singular cubic threefold. The singularities of $\Sigma$ are the $\varphi$-images of planes in $W_3$ meeting $F$ along a quartic curve. For a general  surface $F$, the surface $\Sigma$ is smooth and $\varphi$ is the blowup of the surface $\Sigma$ as in Table \ref{S95}.

\item Let  $L$ be the ample generator of ${\mathrm{Pic}}(X_{2g-2})$. Then we have relations as in Proposition \ref{pro4}.
$$\begin{aligned}
\varphi^\ast L \sim 4\rho^\ast H- E,\quad\quad\quad\quad&\quad\quad D\sim 3\rho^\ast H-E,\\
\rho^\ast H \sim \varphi^\ast L- D, \quad\quad\quad\quad&\quad\quad E\sim 3\varphi^\ast L-4D.
\end{aligned}
$$
\end{enumerate}
\end{prop}

\begin{proof}
 Since $F$ is a scheme-theoretical intersection of quartics, the linear system $|4\rho^\ast H - E|$ is base point free.  
 Hence, the divisor $-K_{\widetilde{\Bbb P^4}} = \rho^\ast H +4\rho^\ast H - E$ is ample, that is, $\widetilde{\Bbb P^4}$ is a Fano fourfold whose rank of ${\mathrm{Pic}}(\widetilde{\Bbb P^4})$ is two.  
 By the Cone Theorem, there exists a Mori contraction $\mu : \widetilde{\Bbb P^4} \dashrightarrow  X_{2g-2}^\prime$ different from $\rho$. 
Let  $D$ be  the proper transform  of the cubic threefold in $\Bbb P^4$ that passes through $F$. We can write $D \sim 3\rho^\ast H-kE$ for some $k > 0$. 
Since  $F$ has $\delta$ transverse double points,  we have 
$$0\le (4\rho^\ast H-E)^3\cdot D=-42(k-1)$$ 
by Lemma \ref{chow1}. Hence, $k=1$ and $(4\rho^\ast H-E)^3\cdot D=0$. 
This means that the divisor class of $4\rho^\ast H-E$ is not ample, and so it yields a supporting linear function of the extremal ray generated by the curves in the fibers of $\mu$. 
Moreover, we can write $4\rho^\ast H-E=\mu^\ast L$ , where $L$ is the ample generator of ${\mathrm{Pic}}(X_{2g-2}') \cong \Bbb Z$.  Then we have $(4\rho^\ast H-E)^4 = 
2g-2$, and $\dim \mu(\widetilde{\Bbb P^4}) = 4$.  
Thus, $\mu$ is birational, and its exceptional locus coincides with $D$. In particular, it is an irreducible divisor. Using the Riemann-Roch theorem and Kodaira vanishing theorem, we obtain the equality $\dim |4\rho^\ast H-E| = g+2$. 
This yields the diagram
$$\xymatrix{&\widetilde{\Bbb P^4}\ar[dr]^{\mu}\ar[dl]^{\rho}\ar[r]^{\varphi}&X_{2g-2}\ar[d]\\
\Bbb P^4\ar@{-->}[rr]^\Psi&&X_{2g-2}^\prime \subset\Bbb P^{g+2},}$$
where $\widetilde{\Bbb P^4} \to X_{2g-2}^\prime \subset\Bbb P^{g+2}$ is given by the linear system $|4\rho^\ast H-E|$, and $$\varphi_{|4\rho^\ast H-E|}:\xymatrix{\widetilde{\Bbb P^4} \ar[r]^{\varphi} & X_{2g-2}\ar[r]& X_{2g-2}^\prime \subset\Bbb P^{g+2}}$$
is the Stein factorization.
\begin{claim}
The variety $X_{2g-2}^\prime$  is a Mukai fourfold with at worst terminal Gorenstein singularities and rank of  ${\mathrm{Pic}} X_{2g-2}^\prime$ is $1$.
\end{claim}
\begin{proof}[Proof of the claim]
Since $\mu$ is a divisorial Mori contraction, $X_{2g-2}^\prime$ has at worst terminal singularities. It follows from ${\mathrm{rank}}\ {\mathrm{Pic}} (\widetilde{\Bbb P^4} )= 2$ that we have ${\mathrm{rank}}\ {\mathrm{Pic}} (X_{2g-2}^\prime) = 1$. Since
$$-K_{\widetilde{\Bbb P^4}} = 5\rho^\ast H  - E=2(4\rho^\ast H -E)-D,$$
we get that $-K_{X_{2g-2}^\prime}=2L$. Hence, $-K_{X_{2g-2}^\prime}$ is an ample Cartier divisor divisible by $2$ in ${\mathrm{Pic}} X_{2g-2}^\prime$. So $X_{2g-2}^\prime$ is a Mukai fourfold.
\end{proof}

The morphism $X_{2g-2}\to X_{2g-2}^\prime \subset \Bbb P^{g+2}$ is given by the linear system $|L| = | -\frac{1}{2} K_{X_{2g-2}}|$. Lemma \ref{Mukaig} implies that this is an isomorphism. In the remaining proof, we identify $X_{2g-2}$ with $X_{2g-2}^\prime $ and $\varphi_{|4\rho^\ast H-E|}$ with $\varphi$.
Hence, we get that $\varphi$ is birational, $\deg X_{2g-2}=2g-2$ and the morphism $\varphi$ contracts the divisor $D$ to an irreducible surface $\Sigma \subset X_{2g-2}$. 

Since ${\mathrm{rank}}\ {\mathrm{Pic}} (\widetilde{\Bbb P^4})= 2$, the exceptional locus of $\varphi$ coincides with $D$, and
$D$ is a prime divisor. Therefore, $\varphi$ has at most a finite number of $2$-dimensional fibers. By the main theorem in \cite{AnW98}, $\Sigma$ has at most isolated singularities. Moreover,  the surface $\Sigma$ as in Table \ref{S95} is normal. Indeed, if  the surface $\Sigma \subseteq \Bbb P^{g-3}$ is not normal, then $\Sigma$ is contained in a $(g-4)$-dimensional subspace and the singular locus of $\Sigma$ is $1$-dimensional, which contradicts the fact  that $\Sigma$ has at most isolated singularities. 

The remaining statements of the proposition follow from the claim below. 
\end{proof}

\begin{claim}
The morphism $\varphi : \widetilde{\Bbb P^4}\to X_{2g-2}$ is the blowup of the surface $\Sigma$, where both $\Sigma$ and $X_{2g-2}$ are smooth if the surface $F$ is general.

\end{claim}
\begin{proof}[Proof of the claim]
Assume that $\Sigma$ or $X_{2g-2}$ are singular, then by \cite[Theorem
2.3]{And85} the extremal $K_{\widetilde{\Bbb P^4}} $-negative contraction $\varphi : \widetilde{\Bbb P^4} \to X_{2g-2}$  would have
a $2$-dimensional fiber, say, $\widetilde{Y}\subset D\subset \widetilde{\Bbb P^4}$. Since $\Sigma$ is normal, by the main theorem of \cite{AnW98}, we have $\widetilde{Y}\cong\Bbb P^2$ and by \cite[Proposition 4.11]{AnW98}, we get $\rho^\ast H|_{\widetilde{Y}}=\mathcal O_{\Bbb P^2}(1)$.
Since $\widetilde{Y}$ is contracted to a point under $\varphi$, we have $(4\rho^\ast H-E)|_{\widetilde{Y}}\sim 0$.
Thus, 
we have $E|_{\widetilde{Y}} =\mathcal O_{\Bbb P^2}(4)$. Therefore, the image $Y=\rho (\widetilde{Y})\subset \rho(D)$  is 
a plane meeting $F$ along a quartic curve, that is, $Y\neq F$ and $Y\cap F\cong \widetilde{Y}\cap E$ is a quartic curve in $Y\cong \Bbb P^2$.  This contradicts that a general surface $F$ is contained in the unique cubic threefold $W_3$ and $W_3$ only contains finitely many planes where non of these planes intersect $F$ in a quartic curve.
Hence, all the fibers of $\varphi$ have dimension less than or equal to $1$. By \cite{And85}, both $X_{2g-2}$ and $\Sigma$ are smooth, and $\varphi$ is the blowup of $\Sigma$. 
\end{proof}

\begin{rem}
The birational maps in Proposition \ref{pro4} and \ref{main10} for the case $g=6$ are also described in \cite[Section 7.3]{DIM15}. They show the relation of Gushel-Mukai fourfolds containing a $\tau$-quadric surface and their associated $K3$ surfaces of genus $6$.  On the other hand, the birational maps in Proposition \ref{pro4} and \ref{main10} for the case $g=8$ are also described in \cite{AlS16}. 

\end{rem}


\section{The geometry of the cubic threefolds $W_3$}\label{geometryCubicThreefold}

We use the notation of the last section. In particular, let $F\subset \Bbb P^4$ be a surface as in Table \ref{S95} and let $W_3$ be the unique cubic threefold containing $F$. We describe the geometry of the cubic threefold $W_3$ appearing in a Sarkisov link for the four different families of Mukai fourfolds as in Theorem \ref{Main}.

\subsection{Case $g=6$}
In the proof of Proposition \ref{main10}, the geometrical construction of a smooth Gushel-Mukai fourfold of genus $6$ containing a $\tau$-quadric surface is achieved as follows. Let $Y \subset \Bbb P^6$ be a general K3 surface of degree $10$. The surface $F \subset\Bbb P^4$ is the image of $Y$ under the internal projection from two points on $Y$. Then $F$ is a singular surface  of degree $8$ and sectional genus $6$ with one ordinary double point. Moreover, $F\subset\Bbb P^4$ lies on {\it a unique cubic threefold $W_3$ which has seven ordinary double points}. 
By Proposition \ref{main10},  there is a smooth Gushel-Mukai fourfold of genus $6$ containing a $\tau$-quadric surface and a birational map $\Psi : {\Bbb P^4} \dashrightarrow  X_{10}$ defined by the linear system of quartics containing $F$. The general Gushel-Mukai fourfold containing a $\tau$-quadric is obtained in this way by Proposition \ref{pro4} (see also \cite[Proposition 7.4]{DIM15}).

\subsection{Case $g=7$:}

We present a geometrical construction of a smooth Mukai fourfold of genus $7$ containing a {\it smooth} cubic scroll surface following Proposition \ref{main10}.  

Let $Y\subset \Bbb P^5$ be a smooth surface given as a blowup of $\Bbb P^2$ in $9$ points in general position by the complete  linear system 
$$H=4L-\sum\limits_{i=1}^{9}E_i,$$
and let $p\in \Bbb P^5\backslash Y$ be a general point not on the surface $Y$. We denote by $F \subset\Bbb P^4$ the singular projection of $Y$ from the point $p$. Moreover, the surface $F$ has degree $7$ and sectional genus $3$, has exactly three ordinary double points and lies on {\it a unique cubic threefold $W_3$ which has seven ordinary double points}. Note that the variety $W_3$ contains either $2$ or $3$ planes (see \cite{FiW89}).

We can reverse this construction. Let $F$ be a surface  of degree $7$ and sectional genus $3$ in $\Bbb P^4$ that is contained in unique cubic threefold $W_3$ with $7$ isolated singularities. 
From the classification of nodal cubic threefolds in \cite{FiW89},  there exists a quadric surface cone $C\subset W_3$ whose vertex is one of singular points of $W_3$.
Thus, the linear system $|2H - C |$ of quadric hypersurfaces in $\Bbb P^4$ passing through $C$ defines a birational map 
$$\Bbb P^4 \dashrightarrow  Q\subset \Bbb P^5,$$
where the image $Q$ is a quadric hypersurface in $\Bbb P^5$. Then, the image of $F$ via this birational map is a smooth surface $Y$ of degree $7$ and sectional $3$ in $\Bbb P^5$. All projective, algebraic varieties of degree $7$ were classified by Ionescu (see \cite{Ion82}) and it follows that $Y$ is a blowup of $\Bbb P^2$ in $9$ points in general position by the complete  linear system $H=4L-\sum_{i=1}^{9}E_i$.

The birational isomorphism $\Psi : {\Bbb P^4} \dashrightarrow  X_{12}$ of Proposition \ref{main10} is given by the linear system of quartics containing $F$. The general Mukai fourfold of genus $7$ containing a cubic scroll surface is obtained in this way by Proposition \ref{pro4}. Note the corresponding Hilbert scheme of smooth surfaces  of degree $7$ and sectional genus $3$ is smooth at the point representing $Y$ of dimension $45$. Thus, this construction depends on $14$ parameters ($9$ for the surface $Y$ and $5$ for $p \in \Bbb P^5$).
The moduli space of the Mukai fourfolds of genus $7$ is of dimension $15$, and we conclude that the family of all Fano fourfolds of genus $7$ obtained by our construction has codimension $1$. 


\subsection{Case $g=8$:} The geometrical construction of Proposition \ref{main10} of a smooth Mukai fourfold of genus $8$ is realized as follows. It follows from the classification of surfaces of degree $7$ in $\Bbb P^4$ (see \cite{Ion82} or \cite{Liv90}) that the smooth surface $F\subset \Bbb P^4$ of  degree $7$ and  sectional genus $4$  is a blowup of $\Bbb P^2$ in $11$ points in general position embedded by the complete linear system
$$H=6L-\sum\limits_{i=1}^{6}2E_i-\sum\limits_{i=7}^{11}E_i.$$
Therefore, the birational isomorphism $\Psi : {\Bbb P^4} \dashrightarrow  X_{14}$ is given by the linear system of quartics containing $F$. The general Mukai fourfold of genus $8$ is obtained in this way by Proposition \ref{pro4}.
Moreover, the surface $F \subset\Bbb P^4$ lies on a unique cubic threefold $W_3$ with the maximal number of $10$ ordinary double points. Such a cubic threefold is unique up to isomorphism. It is called the {\it Segre cubic} and can be explicitly given by the following system of equations:
$$\sum\limits_{i=0}^5x_i=\sum\limits_{i=0}^5x_i^3=0,$$
in $\Bbb P^5$. This variety has many interesting properties (see e.g. \cite[Section 9.4.4]{Dol12}):
\begin{enumerate}[$\bullet$]
\item It is a unique cubic threefold with $10$ isolated singularities.
\item The variety $W_3$ contains exactly $15$ planes forming an ${\mathrm{Aut}}(W_3)$-orbit and one of them
is given by equations
$$x_1+x_2=x_3+x_4=x_5+x_6=0.$$
\item Singular points of $W_3$ form an ${\mathrm{Aut}}(W_3)$-orbit, the point $(1:1:1:-1:-1:-1)$ is one
of them.
\end{enumerate}

\subsection{Case $g=9$:} The  geometrical construction of a smooth Mukai fourfold of genus $9$ containing a sextic del Pezzo surface starts from a general sextic del Pezzo surface $Y \subset \Bbb P^6$ and two general points not on $Y$.  The surface $F \subset\Bbb P^4$ is the (singular) projection of $Y$ from these two points (see \cite[Section 8.4.2]{Dol12}) and the birational isomorphism $\phi : {\Bbb P^4} \dashrightarrow  X_{16}$ is given by the linear system of quartics containing $F$.
Moreover, $F \subset\Bbb P^4$ lies on a \emph{unique cubic threefold $W_3$  with $9$ ordinary double points} and the cubic $W_3$ is a hyperplane section of the cubic fourfold $Z \subset \Bbb P^5$ which can be explicitly given by the following equation: $x_1x_2x_3=y_1y_2y_3$
(see \cite[Proposition 2.2]{ShB12}). The singular locus ${\mathrm{Sing}}(Z)$ consists of nine lines
$$\ell_{ij}=\{x_k=y_l=0\mid k\neq i, l\neq j\},$$
which intersects $W_3$ in the singular locus of $W_3$. Also $Z$ contains nine three-spaces
$M_{ij}=\{x_i=y_j=0\}$
and the intersections of these subspaces with $W_3$ are planes on $W_3$.
The general Mukai fourfold of genus $9$ containing a sextic del Pezzo surface is obtained in this way by Proposition \ref{pro4}. Notice that if $X_{16}$ is a general Mukai fourfold of genus $9$ containing a sextic del Pezzo surface, then $X_{16}$ contains  a quadric $S\subset \Bbb P^3\subset\Bbb P^{11}$ with $c_2(X_{16}) \cdot S = 5$, and this family is already described in \cite[Theorem 2.1]{PrZ16}.


\section{Cylinders in the complement of singular cubic threefolds}\label{CylinderCubics}

\begin{prop}\label{cylinderSC}
Let $W$ be any cubic hypersurface in $\Bbb P^4$ whose singular locus contains  a line $\ell$. 
Then $\Bbb P^4\backslash W$ contains a cylinder. 
\end{prop}

We remark that our proof is similar to \cite[Proposition 3.7 and Lemma 3.10]{KPZ14}. The authors show that the complement of a cubic surface in $\Bbb P^3$ contains a cylinder whenever the surface has a singularity worse than an $A_2$-singularity.  Even more is true.  Note that the complement of a cubic surface with at worst du Val singularities in $\Bbb P^3$ is cylindrical if and only if it is singular (see \cite{CPW16a,Par20}).

\begin{proof}
Let $P$ and $V$ be two vector spaces of dimension $2$ and $5$, respectively, such that  $\ell = \Bbb P(P) \subset \Bbb P(V) = \Bbb P^4$. Furthermore, we pick a generic linear subspace $\Bbb P^2=\Bbb P(V^\prime)\subset \Bbb P(V)$ of codimension $2$ such that the composition $V^\prime \subset V \twoheadrightarrow V/P$ is an isomorphism or, equivalently, that $V^\prime \oplus P = V$.

The linear projection $\Phi:\Bbb P^4 \dashrightarrow \Bbb P^2:=\Bbb P(V^\prime) \cong \Bbb P(V/P)$ from $\ell$ is the rational map sending a point 
$x \in\Bbb P^4 \backslash \ell$ to the unique point of intersection of the linear subspace $\langle x, \ell \rangle \cong \Bbb P^2$ with $\Bbb P(V^\prime) \cong \Bbb P^2$. It is the rational map associated with the linear system $|\mathcal I_\ell (1)| \subset |\mathcal O_{\Bbb P^4}(1)|$ with base locus $\ell \subset \Bbb P^4$. Resolving this map by a simple blowup, we get a  morphism $\rho: \widetilde{\Bbb P^4} \to \Bbb P(V^\prime)$  associated to the complete linear system $|\varphi^\ast \mathcal O(1) \otimes \mathcal O(-E)|$ in the following diagram
 $$\xymatrix{&E=\Bbb P(\mathcal N_{\ell})\ar[dl]\ar@{^{(}->}[r]&\widetilde{\Bbb P^4}\ar[dl]_{\varphi}\ar[dr]^{\rho}&& 
\\
\ell\ \ar@{^{(}->}[r]&\Bbb P^4\ar@{-->}[rr]^{\Phi}&&\Bbb P^2,&}$$
where $\varphi$ is the blowup of $\ell$ and $E$ the exceptional divisor. 
For a point $\lambda \in \Bbb P(V^\prime)$, the fibre $\rho^{-1}(\lambda) \subset \widetilde{\Bbb P^4}$ is the strict transform of $\Pi_\lambda:= \langle \lambda, \ell \rangle  \subset \Bbb P^4$, and is isomorphic to $\Bbb P^2$. Therefore, $\rho: \widetilde{\Bbb P^4} \to \Bbb P(V^\prime)$ is a $\Bbb P^2$-bundle. 

Let $W \subset \Bbb P^4$ be the cubic hypersurface with equation $f \in H^0(\Bbb P^4,\mathcal O(3) \otimes \mathcal I_\ell^2) \subset H^0(\Bbb P^4,\mathcal O(3))$ 
since $\ell$ is contained in the singular locus of $W$. The total transform is the union of the exceptional divisor $E$ and its strict transform $\widetilde W$. 
Moreover, for all $\lambda \in \Bbb P^2$, we have
$$ W\cap  \Pi_\lambda = 2\ell \cup L_\lambda,$$ 
where  $L_\lambda$ is a line. 

Let $\Pi_\infty$ be a hyperplane containing $\ell$ in $\Bbb P^4$, and let $\widetilde {\Pi_\infty}$ be the proper transform  of $\Pi_\infty$ in $\widetilde{\Bbb P^4}$. We consider the open set
$$ \widetilde{U}=\widetilde{\Bbb P^4}\backslash (E \cup \widetilde {\Pi_\infty})\cong \Bbb P^4\backslash \Pi_\infty\cong\Bbb A^4 \subset \widetilde{\Bbb P^4}.$$

Let $h$ be a regular function on $\widetilde{U}$ which defines the affine threefold $\widetilde{U}\cap \widetilde W$. Now, we consider the rational map $\sigma$
 $$\xymatrix{\widetilde{\Bbb P^4} \ar@{-->}[rr]^{\sigma=(\rho,h)}\ar[rd]_{\rho}&&\Bbb P^2\times \Bbb P^1\ar[dl]^{\pr}&\\
 &\Bbb P^2}$$
Its restriction to the open set $\widetilde{U}\backslash (\widetilde{U} \cap \widetilde W)$ is regular, while the restriction to 
a general fiber 
$$\rho^{-1}(\lambda) \backslash (\widetilde W\cup  \widetilde E)\cong \Pi_\lambda \backslash (\ell \cup L_\lambda)\cong \Bbb A^1\times \Bbb A^1_\ast$$ 
of $\rho|_{\widetilde U}$ defines an $\Bbb A^1$-fibration.
Therefore, the rational map $\sigma$ defines  an $\Bbb A^1$-fibration over a Zariski open subset  of $\Bbb P^2 \times  \Bbb P^1$. 

By \cite[Theorem 1]{KaM78} or \cite{KaW85}, there exists a cylinder in $\widetilde{\Bbb P^4}\backslash (\widetilde W\cup \widetilde \Pi_\infty \cup E)$ compatible with this  $\Bbb A^1$-fibration of $\sigma$. Hence, $\Bbb P^4\backslash W$ contains a cylinder, as required. 
\end{proof}

\begin{prop}\label{cylinderCremona}
Let $W$ be a singular cubic hypersurface in $\Bbb P^4$ and let $p$ be an ordinary double point of $W$. 
Then $\Bbb P^4\backslash W$ contains a cylinder. 
\end{prop}

\begin{proof}
Let $W = V(f)$ be generated by a cubic equation $f$, and let $\mathcal{I}_p(1) = (l_1,l_2,l_3,l_4)$ be the linear forms generating the ideal of $p$. We choose a rank $4$ quadric $Q=V(g)$ that is singular in $p$ (that is, $g\in |\mathcal{I}^2_p(2)|$). Then, we consider the following cubic Cremona transformation 
$$
\xymatrix{
\Bbb P^4 \ar@{-->}[rrr]^{|g \cdot l_1, g \cdot l_2, g \cdot l_3, g \cdot l_4, f|} &&& \Bbb P^4.}
$$
The image of $W$ is a hyperplane $P\subset \Bbb P^4$ not containing the image point of $Q$. The total transform of the point $p$ is another rank $4$ quadric $Q'$ which is a cone and its vertex is the image of $Q$. We get 
$$
\Bbb P^4\backslash (W\cup Q) \cong \Bbb P^4\backslash (P\cup Q').
$$
Note that the intersection $P\cap Q'$ is a smooth rank $4$ quadric surface, whence it contains two rulings of lines. Therefore, we may apply Proposition \ref{cylinderSC}, and $\Bbb P^4\backslash(P\cup Q')$ contains a cylinder. Thus, $\Bbb P^4\backslash W$ contains a cylinder, too.

\end{proof}

We end this section by proving our main theorem. 

\begin{proof}[Proof of Theorem \ref{Main}]
Let $X_{2g-2}$ be a Mukai fourfold in one of the families in Theorem \ref{Main}. 
By Proposition \ref{pro4} and \ref{main10}, there exists a cubic threefold $W_3\subset \Bbb P^4$ and a hypersurface in $\varphi(E)\subset X_{2g-2}$ such that 
$$
X_{2g-2}\backslash \varphi(E)\cong \Bbb P^4\backslash W_3.
$$
As shown in Section \ref{geometryCubicThreefold}, the cubic $W_3$ has at least one isolated ordinary double point. By Proposition \ref{cylinderSC} and \ref{cylinderCremona}, it follows that $ \Bbb P^4\backslash W_3$ contains a cylinder. This finishes the proof of Theorem \ref{Main}.
\end{proof}

{\bf Acknowledgement} The authors wish to thank Frank-Olaf Schreyer for his useful suggestions about the topics in this paper. We thank Mikhail Zaidenberg for valuable comments and suggestions on an earlier draft.



\end{document}